\documentclass{amsart}
\usepackage{amsmath,amsfonts,amssymb}
\usepackage{bbm}
  \usepackage{paralist}
 \usepackage[colorlinks=true]{hyperref}
\hypersetup{urlcolor=blue, citecolor=red}

\newtheorem{Thm}{Theorem}
\newtheorem{Cor}{Corollary}
\newtheorem{thm}{Theorem}[section]
\newtheorem{lem}[thm]{Lemma}
\newtheorem{prop}[thm]{Proposition}
\newtheorem{qu}[thm]{Question}

\newcommand\eps{\varepsilon}
\newcommand\RR{\mathbb{R}}
\renewcommand\SS{\mathbb{S}}

\title[Folding maps on spheres] %Use the shortened version of the full title
      {Ergodic properties of folding maps on spheres}

\author[A. Burchard, G.R. Chambers, and A. Dranovski]{}

 \email{almut@math.toronto.edu}
  \email{chambers@math.uchicago.edu}
  \email{a.dranovski@mail.utoronto.ca}

\begin{document}
\maketitle

% Enter the first author's name and address:

\centerline{\scshape Almut Burchard}
\medskip
{\footnotesize
% please put the address of the first author
 \centerline{University of Toronto, Department
of Mathematics}
   \centerline{40 St. George St., Room 6290}
   \centerline{Toronto, ON M5S 2E4, Canada}
} % Do not forget to end the {\footnotesize by the sign }
\medskip

\centerline{\scshape Gregory R. Chambers}
\medskip
{\footnotesize
 % please put the address of the second  and third author
 \centerline{University of Chicago, Department of Mathematics}
   \centerline{5734 S. University Avenue, Room 208C}
   \centerline{Chicago, IL 60637, USA}
}

\medskip
\centerline{\scshape Anne Dranovski}
\medskip
{\footnotesize
% please put the address of the first author
 \centerline{University of Toronto, Department
of Mathematics}
   \centerline{40 St. George St., Room 6290}
   \centerline{Toronto, ON M5S 2E4, Canada}
} 
\bigskip

% The name of the associate editor will be entered by an editorial staff
% "Communicated by the associate editor name" is not needed for special issue.
% \centerline{(Communicated by the associate editor name)}

\begin{abstract}
We consider the trajectories
of points on $\SS^{d-1}$ under sequences of
certain folding maps associated with reflections. 
The main result characterizes 
collections of folding maps that produce dense trajectories.
The minimal number of maps in such a collection 
is $d+1$. 
\end{abstract}

%%%%%%%%%%%%%%%%%%%%%%%%%%%%%%%%%%%%%%%%%%%
\section{Introduction}

The subject of this paper is the reconstruction of
full radial symmetry from partial information.
A function on $\RR^d$ is radial if and only if it is symmetric
under reflection about arbitrary hyperplanes through the origin.
Since the orthogonal group $O(d)$ acts transitively and 
faithfully on $\RR^d$, an equivalent statement is that
the reflections generate $O(d)$.

It is well-known that a finite set of reflections suffices
to generate a dense subgroup of $O(d)$.
In the plane, the composition of two reflections
is a rotation by twice the enclosed angle. If the angle is
incommensurable with $\pi$, then the multiples of the rotation
are dense in the circle.
Likewise in dimension $d>2$, the reflections at
$d$ hyperplanes in general position
generate a dense subgroup of the orthogonal group.
From any given starting point, 
the subset of points that can be reached
by a suitable composition of these reflections
is dense in the centered sphere that contains the point. 
The corresponding random walk is 
almost surely equidistributed on the sphere.

Here we study the corresponding issues for a family of
piecewise isometries that fold the sphere
across a hyperplane onto a hemisphere. We label a folding 
map by the unit vector in the direction
of the target hemisphere, and consider sequences of such
maps indexed by a given set of directions $G\subset\SS^{d-1}$.
The principal aim
of this article is to answer the question:
{\em When can
a dense subset of points in $\SS^{d-1}$ be reached
from an arbitrary starting position by applying
folding maps with directions
chosen from $G$?} The question
came up in prior work on the convergence of random sequences
of symmetrizations~\cite{BF}. Our motivation will be discussed
in Section~\ref{sec:related}.

It turns out that a pair of obvious necessary
conditions on the set of directions
--- one geometric and one algebraic ---
is also sufficient (Theorem~\ref{Thm:dense}).
If a set of directions meets the conditions,
then the only functions on $\RR^d$ that increase under
composition with each of the corresponding folding maps
are radial. Furthermore, the
random walk generated by randomly alternating
these maps is uniquely ergodic 
(Theorem~\ref{Thm:ergodic}). The invariant measure gives
positive mass to all non-empty open
subsets, but does not agree, in general, with the uniform
measure on $\SS^{d-1}$.

%%%%%%%%%%%%%%%%%%%%%%%%%%%%%%%%%%%%%%%%%%%
\section{Main results}
\label{sec:main}

We begin with some definitions.
Let $\SS^{d-1}$  be the standard sphere,
viewed as the set of unit vectors
in $\RR^d$. The geodesic distance
$d(x,y)$ on $\SS^{d-1}$,
given by the enclosed angle between $x$ and $y$,
is related to the chordal distance
in $\RR^d$ by
$$
|x-y|^2= 2-2\cos d(x,y)\,.
$$
The symbol $O(d)$ refers to the group of orthogonal linear 
transformations on $\RR^d$,
and $SO(d)$ to the orientation-preserving
subgroup.  

The uniform probability measure on the sphere
induced by Lebesgue measure on $\RR^d$
is denoted by $\sigma$.  It will be used as a reference
measure throughout the paper. 
By a {\bf null set} we mean a subset
$A\subset\SS^{d-1}$ with $\sigma(A)=0$.
All sets and functions under
consideration are understood to be Borel
measurable, and measures are assumed to be regular
Borel measures.

For a direction $u\in\SS^{d-1}$, let
$R_u x = x-(2x\cdot u) \, u$
be the reflection of a point $x\in\RR^d$ about the orthogonal
hyperplane $u^\perp$.  Clearly, $R_u=R_{-u}$.
Let $H_u=\{x\in \RR^d \mid x\cdot u >0\}$ be the
open positive half-space associated with $u$, and
define
$$
F_u(x)	=\begin{cases}
x\,,     & \text{if}\ x\in H_u\,,\\
R_ux\,, & \text{otherwise.}
\end{cases}
$$
We call $F_u$ the {\bf folding map} in the direction of $u$,
as it folds $\RR^d$ across the crease at $u^\perp$ onto the 
closed nonnegative half-space.
It is nonlinear and idempotent $(F_u^2=F_u$).
Since $R_u$ is isometric, $F_u$ is non-expansive
$$
|F_u(x)-F_u(y)|\le |x-y|\,,
$$
with strict inequality if $x$ and $y$ lie on opposite sides
of  $u^\perp$.

The {\bf two-point symmetrization} of a real-valued function $\phi$ on
$\SS^{d-1}$ is the equimeasurable rearrangement
defined by
$$
S_u\phi (x)	=\begin{cases}
\max\{\phi(R_u x),\phi(x)\}\,,     & \text{if}\ x\in H_u\,,\\
\min\{\phi(R_u x),\phi(x)\}\,, & \text{otherwise.}
\end{cases}
$$
Note that $S_u\phi=\phi$, if and only if $\phi\circ F_u\ge \phi$.

Let $(u_n)_{n\ge 1}$ in $\SS^{d-1}$ be a sequence of directions
in $\SS^{d-1}$. We study the {\bf trajectory} of a 
point $x$ in $\RR^d$ or $\SS^{d-1}$ under the 
corresponding sequence of folding maps, given by
\begin{equation}
\label{eq:def-trajectory}
x_0=x\,,\qquad x_n= F_{u_n} x_{n-1}\ \text{for}\ n\ge 1\,.
\end{equation}
Under what conditions on the sequence $(u_n)$
are these trajectories dense in $\SS^{d-1}$?
Our main result shows that it suffices to
choose the directions from a 
small subset $G\subset \SS^{d-1}$.

\begin{Thm} [Transitivity] \label{Thm:dense}
Let $G\subset \SS^{d-1}$ be a set of directions.  If
\begin{itemize}
\item [(C1)] the open half-spaces $H_u$ with $u\in G$ cover $\SS^{d-1}$,
and
\item [(C2)] the reflections $\{R_u\mid u\in G\}$
generate a dense subgroup of $O(d)$,
\end{itemize}
then there exists a sequence $(u_n)_{n\ge 1}$ in $G$
such that for every starting point $x\in\SS^{d-1}$,
the trajectory defined by 
Eq.~\eqref{eq:def-trajectory}
is dense in $\SS^{d-1}$.
\end{Thm}

Both assumptions on $G$ are clearly necessary for
the existence of even one dense trajectory.
In Section~\ref{sec:C1C2}, 
we show that the geometric condition (C1) is equivalent to
the origin lying in the interior of the
convex hull of $G$  in $\RR^d$~(Proposition~\ref{prop:cover}(c)).
In particular, $G$ must contain
at least $d+1$ directions that span $\RR^d$.
Condition (C2) can be replaced by an explicit 
sufficient condition (Proposition~\ref{prop:transitive}). 
Any set of directions that spans $\RR^d$ can be augmented by one more
direction to satisfy both conditions
(Proposition~\ref{prop:trans-reduction}).

As for the conclusion, we claim
that there is a sequence of directions that produces
dense trajectories simultaneously for all starting
points on the sphere. 
A useful consequence of
Theorem~\ref{Thm:dense} is the following characterization of
radial functions in terms of two-point symmetrizations.
\begin{Cor} \label{cor:radial}
Let $\phi$ be a continuous function on $\RR^d$, and
let $G\subset\SS^{d-1}$ satisfy (C1) and (C2).
Then
$$
S_u\phi = \phi\ \text{for all}\ u\in G
\quad \Longleftrightarrow\quad \phi\ \text{is radial}\,.
$$
\end{Cor}

The corollary improves upon the 
statement that for any non-radial function
there exist two-point symmetrizations that push 
the function towards its symmetric decreasing rearrangement.
This observation appears as a key step in Baernstein and Taylor's 
proof of Riesz' rearrangement inequality on the sphere~\cite[Theorem 2
and p. 252]{BT}, and many other classical results
(see also \cite[p. 226]{Bec}, \cite[Lemma 6.5]{BrSo}, 
~\cite[Lemma 2.8]{BS}), and \cite[Lemma 3.11]{Mor}.
The corollary extends directly to measurable functions
modulo null sets.

For the proof of Theorem~\ref{Thm:dense} in % typo on page 3 fixed
Section~\ref{sec:dense}, we consider
subsets of the sphere that are positively
invariant under the folding maps indexed by
a set $G$. Assuming $G$ satisfies (C1), every positively
invariant subset of the sphere is invariant, up to
null sets, under the corresponding reflections
(Lemma~\ref{lem:inv}). Condition (C2) 
precludes the existence of non-trivial invariant 
compact subsets of the 
sphere (Lemmas~\ref{lem:01} and~\ref{lem:compact}).
For any given starting point,
the union of all possible trajectories under
sequences of folding maps indexed by $G$ 
is a positively invariant subset. 
We argue that this union is dense,
and then select the desired dense trajectory
from it. Our construction yields a sequence of directions 
$(u_n)$ with the following property: For
any $\eps>0$, there exists a number $N$ such that
every trajectory intersects every ball
of radius $\eps$ in $\SS^{d-1}$
within the first $N$ steps, see Eq.~\eqref{eq:unif-dense}.

In the last two sections,
we consider sequences of directions $(U_n)_{n\ge 1}$
that are chosen independently at random from
a probability distribution $\mu$ on $\SS^{d-1}$,
that~is,
$$
P(U_n\in A_n, \text{for}\ n=1,\dots, N) = \prod_{n=1}^N \mu(A_n)
$$
for any Borel sets $A_1,\dots, A_N\subset\SS^{d-1}$.
The trajectory defined by
\begin{equation}
\label{eq:def-RW}
X_0=x\,,\qquad X_n= F_{U_n} X_{n-1} \ \text{for}\ n\ge 1
\end{equation}
will be called the {\bf random walk}
generated by $\mu$ with starting point $x$.
We are most interested in
examples where $\mu$ is concentrated on a finite 
set of directions.  

Let $G\subset\SS^{d-1}$ be the {\bf support} of $\mu$, that is,
the smallest closed subset of full $\mu$-measure.
If $G$ satisfies the hypotheses of Theorem~\ref{Thm:dense},
then the random walk
is almost surely dense (Corollary~\ref{prop:RW-dense}).
Our second theorem strengthens this observation.
To state the result, define a linear transformation 
on measures by
\begin{equation}\label{eq:T-sharp}
T_\mu \#\nu \, (A)
= \int_{\SS^{d-1}}
\nu(F_u^{-1}(A))\,d\mu(u)\,
\end{equation}
for $A\subset\SS^{d-1}$, and every Borel measure $\nu$.
By definition, $T_\mu\#$
maps the  probability distribution of $X_n$
to the distribution of $X_{n+1}$ for each $n\ge 0$.

\begin{Thm}[Unique ergodicity]\label{Thm:ergodic}
Let $\mu$ be a regular Borel probability measure
on $\SS^{d-1}$ whose support satisfies
{\rm (C1)} and {\rm (C2)}, and
let 
$T_\mu\#$ be given by Eq.~\eqref{eq:T-sharp}.
There is a unique regular Borel probability measure $\rho$
on $\SS^{d-1}$ with $T_\mu\#\rho = \rho$.
The support of $\rho$ is the entire
sphere $\SS^{d-1}$.
\end{Thm}

We refer to $\rho$ as the {\bf invariant measure} for the random walk.
By uniqueness, it assigns zero or one 
to every invariant subset. The invariant measure
governs the behavior of the random
trajectories $(X_n)$ through Birkhoff's ergodic theorem.
\begin{Cor}
\label{cor:Birkhoff}
Let $A\subset\SS^{d-1}$ be a Borel set.
Under the assumptions of Theorem~\ref{Thm:ergodic},
$$%\begin{equation} \label{eq:Birkhoff}
\lim_{N\to\infty} \frac{1}{N} \left|\bigl\{
n\le N \mid X_n\in A\bigr\}\right| = \rho(A)
$$%\end{equation}
almost surely for $\rho$-almost every starting point $x\in\SS^{d-1}$.
\end{Cor}

Thus, typical trajectories are
equidistributed according to the invariant measure~$\rho$. Since 
$\rho$ gives positive measure
to every non-empty open subset,
this provides
another proof that almost all trajectories are dense.

The proof of Theorem~\ref{Thm:ergodic}
is given in Section~\ref{sec:random}.
To construct the invariant measure,
we study the Markov chain associated
with the random walk from Eq.~\eqref{eq:def-RW}
through its adjoint action on functions, 
defined by
$$
\phi_n(x) = E_x(\phi(X_n))\,,
$$
where $x$ is the starting point of the random walk.
Under the assumptions of Theorem~\ref{Thm:ergodic},
we show that for each continuous function $\phi$
on the sphere, the sequence $\phi_n$ converges
uniformly to a constant $\bar \phi$ 
(Proposition~\ref{prop:ergodic}). 
Writing $\bar\phi = \int \phi\, d\rho$ for a suitable
measure $\rho$ on the sphere, we
show that $\rho$ is uniquely determined
by the map $\phi\mapsto \bar\phi$, and invariant under the
transformation $T_\mu\#$. Since $\bar\phi$ lies strictly between
the maximum and minimum of $\phi$ unless $\phi$ is constant,
we find that $\rho(A)>0$ for every non-empty open set $A$.

In Section~\ref{sec:inv-meas}, we
consider the
relationship between the invariant measure
$\rho$ and the uniform 
probability measure $\sigma$ on $\SS^{d-1}$.
If the directions $(U_n)$ are
uniformly distributed on the sphere, then $\rho$ is uniform as well.
In that case, Corollary~\ref{cor:Birkhoff} implies
that the random trajectories are almost surely
equidistributed on the sphere.  More generally,
the uniform measure is invariant if and only
if the distribution of $U_n$ is
even under the antipodal map $u\mapsto -u$
(Proposition~\ref{prop:even}).
It is an open question whether
the invariant measure
is always mutually absolutely continuous with respect to
the uniform measure.

%%%%%%%%%%%%%%%%%%%%%%%%%%%%%%%%%%%%%%%%%%%
\section{Motivation and related work}
\label{sec:related}

Symmetrizations are equimeasurable 
rearrangements of sets and functions which are commonly 
used for proving that
certain optimization problems have radially symmetric
solutions.  For instance, the sharp
constants in the Young, Sobolev,
and Hardy-Littlewood-Sobolev inequalities are
assumed among radial functions, the ground state of the
hydrogen atom is symmetric decreasing, and the membrane 
of given area with the lowest fundamental frequency
is shaped like a disk.

A standard technique for
establishing geometric inequalities
is to approximate full radial symmetrization 
by a sequence of simpler symmetrizations.
Two-point symmetrization has been used in this way
to prove the isoperimetric inequality on spheres~\cite{Ben},
and sharp inequalities for path integrals~\cite{BS,Mor,PS}. 
The convergence of random sequences of 
symmetrizations has received some attention in the
literature, most notably in the work of
Klartag~\cite{K} on Steiner symmetrizations of
convex sets. Convergence of two-point
symmetrizations is less well studied. Very
recently, De Keyser and Van Schaftingen have
considered random symmetrization processes
with time correlations~\cite{dKvS}. Open
questions in the area include precise conditions 
for convergence, and bounds on the rate of convergence.

To explain the relationship with
our results, let $\phi$ be a nonnegative continuous
function with compact support on $\RR^d$.
Consider a sequence $(\phi_n)$ of functions obtained from $\phi$
by iterated Steiner symmetrization according to
a sequence of directions. Any limit point $\psi$ of
$(\phi_n)$ has at least some reflection
symmetries, that is, $\psi\circ R_u=\psi$ for all
$u$ in a non-empty set of directions $G\subset\SS^{d-1}$.
If $G$, which is determined by the sequence of Steiner symmetrizations,
satisfies (C2), it follows that $\psi$ is radial.
Then one can identify 
$\psi$ as the symmetric decreasing rearrangement of $\phi$ 
and conclude that the entire sequence of symmetrizations 
converges to $\psi$~\cite[Corollary 2.3b]{BF}.
This provides a sharp condition for an i.i.d. 
sequence of Steiner symmetrizations
to converge to the symmetric decreasing rearrangement.

If, instead, $\phi_n$ is obtained
by iterated two-point symmetrization
at hyperplanes in $\RR^d$ that
do not contain the origin, then
any limit point $\psi$ of the sequence
has the property that $S_u\psi=\psi$, that
is, $\psi\circ F_u\ge \psi$ 
for all $u$ in a some non-empty set of directions $G\subset\SS^{d-1}$.
If $G$ satisfies (C1) and (C2), then according to
Corollary~\ref{cor:radial}, the function $\psi$ is radial.
In combination with~\cite[Theorem 2.2]{BF} 
this yields a sharp condition for the convergence of 
i.i.d. sequences of two-point symmetrizations, the first
result of this type.
Our hope is that a full analysis of the random walk 
defined in Eq.~\eqref{eq:def-RW}
would shed light on the best achievable
rate of convergence. We note that
for non-convex sets, the
strongest known lower bound on the rate of
convergence of Steiner symmetrizations was proved
by comparison with two-point 
symmetrization~\cite[Corollary 5.4]{BF}.

There are many known results analogous to Theorem~\ref{Thm:dense}
for other collections of maps associated with
linear isometries of spheres.  For example, 
Crouch and Silva Leite~\cite{Crouch-Leite} found 
pairs of one-parameter subgroups of $SO(d)$ 
which %uniformly finitely 
generate all of $SO(d)$, and
produced upper bounds on the number of elements from 
each subgroup required to generate any element
(see also Levitt and Sussmann~\cite{LS} for related results).
An example of this phenomenon
is the Euler angles decomposition, a formula for writing 
any element of $SO(3)$ as a product of three rotations 
about the $x$ and $y$ axes. Rosenthal~\cite{Rosenthal} has 
established bounds on the rate of convergence to the steady state
for random walks generated by conjugacy classes of planar rotations
on $SO(d)$.  Porod~\cite{P} has obtained analogous results
for random walks on $O(d)$, $U(d)$ and $Sp(d)$
generated by the conjugacy class of reflections.

All of the above results can be translated 
into statements about trajectories on
$\SS^{d-1}$ via the standard actions of the groups.
Historically, they were preceded by results on the symmetric
group $S_n$, with transpositions playing the role
of reflections.
Dixon~\cite{Dixon} proved that the probability that 
two randomly selected elements of the symmetric group 
$S_n$ generate the whole group approaches $3/4$ as 
$n \rightarrow \infty$. Diaconis and 
Shahshahani~\cite{DS} studied random walks
on the symmetric group generated by random transpositions
and developed methods for estimating the rate of convergence.

Since a folding map agrees with the identity map on
one half-space, and with a reflection on its complement,
it is a piecewise isometry.  The dynamics
of piecewise isometries have been studied in various contexts by
numerous authors, for example see the survey by Goetz~\cite{G}.
Different from interval exchange maps, our folding maps are neither one-to-one
nor onto.

On an abstract level, our results are motivated by
classical theorems about directed graphs, and more generally
Markov chains on discrete state spaces.
A graph is {\bf transitive} if every vertex $x$ can be connected
to every other vertex $y$ by a path of directed edges.
Finding a minimal subset of edges that still
connects all vertices is known to be a hard problem, 
particularly if the graph has cycles.
As a practical alternative,
Aho, Garey, and Ullmann developed the theory of
transitive reductions,
which are graphs on the same vertex set
with the smallest possible number of edges~\cite{AGU}. 
In our setting, the points $x\in \SS^{d-1}$ play the role of 
vertices, and a pair $(x,u)$ with $u\in G$ plays the role of 
a directed edge connecting $x$ to $F_ux$. 
Edge paths correspond to
trajectories, and transitivity means that
all trajectories are dense.
Theorem~\ref{Thm:dense} gives necessary and sufficient conditions
for transitivity, Proposition~\ref{prop:trans-reduction}
constructs a transitive reduction,  and
Proposition~\ref{prop:periodic}
establishes the presence of cycles. 
The random walk defined in Eq.~\eqref{eq:def-RW}
corresponds to the Markov chain defined by
a weighted directed graph,
with $\mu$ playing the role of the edge weights.
Theorem~\ref{Thm:ergodic} and Corollary~\ref{cor:Birkhoff}
concern the steady-state of the Markov chain.

Another parallel can be drawn between 
Eq.~\eqref{eq:def-RW} and shift dynamics.
Suppose that the probability density $\mu$ is
concentrated on a finite set $G$
of directions.  For a given initial point,
$x\in\SS^{d-1}$,
if we omit from the underlying sequence
of directions those terms where the folding map
acts as the identity (rather than by 
reflection),
we obtain a random sequence in $G$
where only certain transitions are admissible.
The conditions for admissibility,
however, depend on the current
location of the random walk in $\SS^{d-1}$; 
for this reason the sequence of 
directions is not simply a subshift of finite type.

%%%%%%%%%%%%%%%%%%%%%%%%%%%%%%%%%%%%%%%%%%%
\section{Conditions (C1) and (C2)}
\label{sec:C1C2}

We now discuss the hypotheses of the main theorems.
The geometric condition (C1) can be expressed in a number of 
different forms.  Note that only parts (c) and (d) of 
Proposition~\ref{prop:cover} are needed subsequently.

\begin{prop}
\label{prop:cover}
Let $G\subset \SS^{d-1}$.
The following are equivalent:
\begin{enumerate}
\item [(a)]
The open half-spaces $H_u$ with $u\in G$ cover $\SS^{d-1}$;
\item [(b)] $G$ is not contained in any closed hemisphere
of $\SS^{d-1}$;
\item [(c)]the convex hull of $G$ in $\RR^d$
contains the origin in its interior.
\end{enumerate}
Let $G$ be the support of a Borel probability measure $\mu$.
Then the above conditions hold if and only if
\begin{itemize}
\item[(d)] $0<\mu(H_x)<1$ for all $x\in\SS^{d-1}$.
\end{itemize}
\end{prop}

\begin{proof} $(a)\Rightarrow (b)$: If $G$ satisfies (a), then
for every
$x\in \SS^{d-1}$, there is a direction $u\in G$
such that $-x\in H_u$. This means that $x\cdot u<0$, implying that
$G$ is not contained in the
closed hemisphere $\{u\in\SS^{d-1}\mid x \cdot u\ge 0\}$.
Since $x$ was arbitrary, this shows~(b).

$(b)\Rightarrow(c)$:
If (c) does not hold, then there is a hyperplane
through the origin that does not meet the interior
of the convex hull of $G$. Therefore, $G$ is contained
in a closed half-space, which intersects $\SS^{d-1}$ in a closed
hemisphere, contradicting (b).

$(c)\Rightarrow(a)$: Let $C$ be the convex hull of $G$ in $\RR^d$.
Given a point
$x$ in $\SS^{d-1}$, consider
the linear functional defined by $\ell(y)=x\cdot y$
on $\RR^d$. If the origin is an interior
point of $C$,
then $\ell$ takes both positive and negative values
on $C$.
Since $\ell$ assumes its maximum at an extreme point of $C$,
and the extreme points of $C$ are contained in the closure
of $G$, we conclude that $\ell(u)>0$ for at least
one $u\in G$. Thus $x\in H_u$, establishing (a).

$(a)\Leftrightarrow(d)$: Let $\mu$ be a Borel probability measure on
$\SS^{d-1}$, and $x\in \SS^{d-1}$.
Assuming (a), the support of $\mu$ contains a direction
$u$ with $x\in H_u$, that is, $u\in H_x$.
Since $H_x$ is open, it follows that $\mu(H_x)>0$,
proving the first inequality in~(d).
Replacing $x$ with $-x$ shows that
$1-\mu(H_x) \ge \mu(H_{-x}) >0$, proving the
second inequality in (d).
The converse implication is obvious.
\end{proof}

Denote by
$$
\langle G\rangle :=
\left\{ R_{u_n}\dots R_{u_1}\in O(d)
\ \vert\ n\ge 0, u_1,\dots, u_n\in G\right\}
$$
the subgroup generated by the reflections $R_u$ 
at hyperplanes with unit normals $u\in G$.
The algebraic condition (C2) can fail in three ways:
If $\langle G\rangle$
lies in a lower-dimensional subgroup, if it splits into
two subgroups that act on orthogonal
subspaces, or if $\langle G\rangle$ is
a finite Coxeter subgroup of $O(d)$.
Thus, assumptions (1) and (2) in the
following proposition are also necessary.
In dimension $d>2$, assumption (3) is stronger than
necessary, because only integer multiples
of $\frac\pi 3$, $\frac\pi 4$, and $\frac \pi 5$
can appear as angles between
elements of a finite %irreducible 
Coxeter subgroup that acts on $\RR^d$~\cite[Theorem~9 and 
Proof of Lemma 4.2]{Coxeter}. 
For $(d+1)$-element sets $G\subset\SS^d$, a precise condition
for generating finite Coxeter groups 
was obtained by Felikson~\cite{F}; for more general sets of 
reflections in $\RR^d$ this is an open problem.
A set of conditions slightly stronger than (1), (2), and (3) 
was used by Eggleston~\cite[proof of Theorem~46]{E}
and Klain~\cite[after Corollary 5.4]{Klain}
to establish convergence to balls for
sequences of Steiner symmetrizations.

\begin{prop}
\label{prop:transitive}
Let $G\subset\SS^{d-1}$. If
(1) $G$ spans $\RR^d$,
(2) $G$ is not a union of
two non-empty mutually orthogonal subsets,  and
(3) not all angles between directions are
commensurable with $\pi$,
then the subgroup $\langle G \rangle$ is dense in $O(d)$.
\end{prop}

\begin{proof} We argue by induction over the dimension.
For $d=1$ there is nothing to show. For
$d=2$, by (1) and (3) there are directions
$u,v\in G$ such that the angle 
$d(u,v)$ is incommensurable with $\pi$. 
Since the composition
$R_uR_v$ generates a dense subgroup of rotations in 
$SO(2)$, it follows that $\langle\{u,v\}\rangle$ is dense
in $O(2)$.

Suppose now that the proposition holds in dimension $d-1$, where
$d>2$.  If $G\subset \SS^{d-1}$ satisfies (1), (2), and (3),
then it contains a subset $G'$ that spans some
hyperplane $v^\perp$ in $\RR^d$ and also satisfies (2) and (3).
Let $S_v$ be the orthogonal group on $v^\perp$.
By the inductive hypothesis, $\langle G'\rangle $
is dense in $S_v$. Note that $S_v$ has a unique pair
of fixed points at $\pm v$, and
acts transitively on the unit sphere in $v^\perp$.

By (1) and (2) there is a direction
$u\in G$ that is linearly independent but not orthogonal to $G'$.
In particular, $w=R_uv$ is  linearly independent of $v$.
Therefore the conjugate subgroup
$S_w=R_u S_vR_u^{-1}$ is different from $S_v$.
Intersecting $S_v$ and $S_w$ with $SO(d)$, we
obtain two distinct subgroups
conjugate to $SO(d\!-\!1)\times \{1\}$ in $\overline{\langle G\rangle}$.
But $SO(d)$ has no non-trivial compact subgroup
that properly contains a copy of 
$SO(d\!-\!1)$~\cite[Lemma 4]{MS}.
It follows that $\overline{\langle G\rangle}$ contains $SO(d)$.
Since it also contains orientation-reversing elements,
$\overline{\langle G\rangle}=O(d)$.
\end{proof}

Our next results concern finite subsets of $\SS^{d-1}$.

\begin{prop}
\label{prop:trans-reduction} 
The minimal number of directions in $G\subset\SS^{d-1}$ required
to satisfy (C1) and (C2) 
is $d+1$.
\end{prop}

\begin{proof}
Condition (C1) requires
at least $d+1$ points, the minimal number of
extreme points for a convex set with interior in $\RR^d$.

To construct a
set of $d+1$ points that satisfies
both (C1) and (C2), we start from a basis 
of unit vectors $u_1,\dots, u_d$
for $\RR^d$. The convex cone 
generated by this basis has non-empty interior. 
We choose a unit vector $u_{d+1}$ whose antipode $-u_{d+1}$ lies 
in the interior of the cone, and, moreover, $u_{d+1}$ encloses an 
angle with $u_1$ that is not a rational
multiple of $\pi$, and set
$G=\{u_1,\dots, u_{d+1}\}$.
Then (C1) holds by
Proposition~\ref{prop:cover}(c),
and (C2) by Proposition~\ref{prop:transitive}.
\end{proof}

We do not have a good characterization
of minimal subsets (under inclusion) satisfying
(C1) and (C2). Such minimal subsets may
contain more than $d+1$ directions.
For example, both conditions hold for
$G=\{u_1,-u_1, \dots, u_d, -u_d\}$,
where $u_1$ and $u_2$ enclose an angle incommensurable with $\pi$
and $u_1,\dots, u_d$ form a basis of unit vectors in $\RR^d$,
but every proper subset is contained in a 
closed hemisphere and thus fails (C1). 

There are sets in $S^1$
satisfying (C1) and (C2) that have
no minimal subsets.  For example, the set
$G=\{ \pm e^{i\pi/k}\mid k\ge 1\}\subset\SS^1$ satisfies
(C1) and (C2), and so does every infinite even subset,
while every finite
subset generates a finite dihedral 
subgroup of $O(1)$. Our next result shows that
in higher dimensions, minimal subsets always exist, 
and are finite.

\begin{prop}\label{prop:min-equivalent}
Let $d>2$. If $G\subset\SS^{d-1}$
satisfies (C1) and (C2),
then there is a finite subset $G'\subset G$ that also
satisfies both conditions.
\end{prop}
\begin{proof} If $G$ is finite, there is nothing to show.
If $G$ is infinite, we choose a small $d$-dimensional 
simplex  inside the convex hull of $G$ in $\RR^d$
that contains
the origin in its interior.  By the Krein-Milman 
theorem, each vertex  of the simplex is a convex
combination of a finite subset of $G$.
The convex hull of the union of these subsets 
contains the origin in its interior.
By Proposition~\ref{prop:cover}(c), 
it satisfies condition (C1).  To obtain
$G'$, we include up to two more
elements of $G$ in the union in such a way that
$G'$ does not split into mutually orthogonal
subsets, and that it contains a pair
of vectors enclosing an angle that is not an integer
multiple of $\frac\pi 3$, $\frac \pi 4$, or $\frac \pi 5$.
In dimension $d>2$, this implies that
$G'$ is not contained in a finite Coxeter group.
Thus (C1) and (C2) both hold for~$G'$.
\end{proof}

%%%%%%%%%%%%%%%%%%%%%%%%%%%%%%%%%%%%%%%%%%%
\section{Invariant subsets and dense trajectories}
\label{sec:dense}

This section is dedicated to the proof of Theorem~\ref{Thm:dense} and
Corollary~\ref{cor:radial}. 
Let $F_u$ and $R_u$ be the folding map
and the reflection defined by a direction
$u\in\SS^{d-1}$.
By definition, a set $A\subset\RR^d$
is {\bf positively invariant} under $F_u$ if
$$
F_u(A)\subset A\,;
$$
if $R_uA=A$ then $A$ is {\bf invariant} under $R_u$.
We say that $A$ is almost positively invariant 
if $F_uA\setminus A$ is a null set with respect to
the uniform probability measure $\sigma$;
if the symmetric difference $A \bigtriangleup R_uA$
is a null set then $A$ is almost invariant.

We consider sets that are positively invariant under
the folding maps indexed by a non-empty
set of directions $G\subset\SS^{d-1}$. 
Under condition (C1),
positively invariant subsets are almost invariant.

\begin{lem}\label{lem:inv}
Let $G\subset \SS^{d-1}$ be a subset that is not
contained in any closed hemisphere.
If $A\subset \SS^{d-1}$ is almost
positively invariant under $F_u$ for all $u\in G$, then
it is almost invariant under $R_u$ for all $u\in G$.
\end{lem}

\begin{proof}
If $A$ is almost positively invariant under $F_u$, then
$(R_u A\cap H_u)\setminus A$ is a null set.  In that case,
$$
\int_A x\cdot u\, d\sigma(x)
= \int_{A\cap H_u} x\cdot  u\, d\sigma(x)
- \int_{(R_uA)\cap H_u} x\cdot  u\, d\sigma(x)
\ge 0\,.
$$
Equality holds only if $(A\cap H_u)\setminus R_uA$
is a null set, in which case $A$ is almost invariant
under $R_u$.

Fix $u_0 \in G$. By Proposition~\ref{prop:cover},
the origin is an interior point of the convex hull of $G$.
This means that for $\eps>0$ sufficiently small,
$-\eps u_0$
can be represented as a convex combination
\begin{equation}
\label{eq:convex}
-\eps u_0= \sum_{i = 1}^n \alpha_i u_i\,
\end{equation}
with $u_i \in G$ and $\alpha_i \geq 0$ for $1 \leq i \leq n$.
It follows that
$$
\eps \int_A x\cdot u_0\, dx+ \sum_{i=1}^n
\alpha_i \int_A x\cdot u_i\, d\sigma(x) = 0\,.
$$
By the positive invariance of $A$, all summands are
nonnegative.  Since $\eps>0$, the integral
$\int x\cdot u_0 \, d\sigma(x)$ vanishes,
and therefore $A$ is almost invariant under $R_{u_0}$.
\end{proof}

\bigskip Under condition (C2), the sphere
has no non-trivial almost invariant subsets.

\begin{lem}\label{lem:01}
Let $G\subset\SS^{d-1}$. If
$\langle G \rangle$ is dense in $O(d)$,
then the only subsets of $\SS^{d-1}$
that are almost invariant under the reflections 
$\{R_u\mid u\in G\}$ are null sets and their complements.
\end{lem}

\begin{proof} The claim can be proved directly, 
but we give a shorter proof which uses spherical harmonics.
Assume that $A\subset \SS^{d-1}$ is almost invariant
under $R_u$ for all $u\in G$.
Consider the indicator
function of $A$ as an element of $L^2(\SS^{d-1})$, and
expand it in spherical harmonics as
$$
\mathbbm{1}_A(x) = \sum_{k=0}^\infty Y_k(x)\,.
$$
Here, $Y_k$ is orthogonal projection of $\mathbbm{1}_A$
onto the spherical harmonics of degree
$k$, and the series converges in the mean-square sense.
Since this representation is unique,
each $Y_k$ is invariant under
composition with $R_u$ for all $u\in G$. Therefore,
$Y_k$ is invariant under $\langle G\rangle$, and, by continuity,
under the entire group $O(d)$. But for $k>0$,
the action of $O(d)$ on the spherical harmonics of degree $k$
fixes only the zero polynomial. It follows that
$Y_k=0$ for all $k>0$,
and the constant function $Y_0=\sigma(A)$
agrees $\sigma$-almost everywhere
with $\mathbbm{1}_A$.
\end{proof}

\bigskip\noindent Combining Lemmas~\ref{lem:inv}
and~\ref{lem:01}, we conclude that there are no
non-trivial compact positively invariant subsets.

\begin{lem}\label{lem:compact} Assume $G\subset\SS^{d-1}$
satisfies {\rm (C1)} and {\rm (C2)}. Then
no compact subset (and no open subset) of $\SS^{d-1}$
other than $\emptyset$ and $\SS^{d-1}$
is positively invariant under all foldings $\{F_u\mid u\in G\}$.
\end{lem}

\begin{proof} Let $A\subset \SS^{d-1}$ be a non-empty
compact set that is positively invariant under $F_u$ for
all $u\in G$.  By Lemma~\ref{lem:inv}, $A$
is almost invariant under $R_u$ for all $u \in G$, and by
Lemma~\ref{lem:01} either $A$ or its complement
is a null set.
We want to exclude the first alternative.

For every $\eps>0$,
$$
A_\eps=\{x\in\SS^{d-1}\ \vert\ d(x,A)\le \eps\}
$$
is a compact set of positive measure.
Since $A$ is positively invariant and the maps $F_u$ are
non-expansive, $A_\eps$ is
positively invariant as well. By Lemmas~\ref{lem:inv}
and Lemma~\ref{lem:01} its complement 
in $\SS^{d-1}$ is a null set; in particular, $A_\eps$ is
dense in $\SS^{d-1}$.
By compactness, $A_\eps=\SS^{d-1}$, and hence
$A=\bigcap A_\eps=\SS^{d-1}$.

If, on the other hand, $A$ is a non-empty open
set that is positively invariant under $G$,
then its complement is a compact set that is
positively invariant under $-G$. By the first part of the proof,
it is empty.  \end{proof}

\begin{proof}[Proof of Theorem~\ref{Thm:dense}.]
For $x\in\SS^{d-1}$, consider the orbit
$$
G_\star x =\{F_{u_n}\dots F_{u_1}x \vert\   n\ge 0,
u_1,\dots u_n\in G\}\,.
$$
By definition, $G_\star x$ contains all trajectories of
$x$ under sequences $(F_{u_n})$ with directions $u_n\in G$.
Since $G_\star x$ is positively invariant under $F_u$ for $u\in G$, 
its topological closure
is a positively invariant non-empty compact subset of $\SS^{d-1}$.
By Lemma~\ref{lem:compact},  the only such subset is the entire
sphere. It follows that $G_\star x$ is dense in $\SS^{d-1}$.
Hence there exists for every $\eps>0$ and every $x,y\in\SS^{d-1}$
a finite sequence of directions $u_1,\dots u_n$ in $G$
such that $d(F_{u_n}\dots F_{u_1}x, y)<\eps$.

We claim that the sequence can be chosen independently
of $x$ and $y$. 
In fact, for every $\eps>0$ there
is a finite sequence $u_1,\dots u_N$ in $G$ such that
\begin{equation}
\label{eq:unif-dense}
\min_{n\le N}
d(F_{u_n}\dots F_{u_1}x, y)<\eps\qquad \forall x,y\in\SS^{d-1}\,.
\end{equation}
The sequence is constructed 
by concatenating a finite number of 
shorter segments $\mathcal{S}_1,\dots, \mathcal{S}_K$ as follows.

Given $\eps>0$, cover $\SS^{d-1}$
by finitely many open balls $B_1,\dots, B_K$ of radius $\eps/3$ centered
at $c_1,\dots, c_K$.  To construct $\mathcal{S}_1$,
choose a finite sequence of directions in $G$ such that
the corresponding
trajectory starting at $c_1$ visits the ball $B_2$, and then
extend that sequence so that the trajectory
visits each of the balls $B_1,\dots B_K$.
The segments $\mathcal{S}_k$ for $1 < k \leq K$
are constructed inductively.
Assuming $\mathcal{S}_1,\dots, \mathcal{S}_{k-1}$
have already been chosen, 
let $y_k$ be the final point of the trajectory of $c_k$ under 
$\mathcal{S}_1, \dots, \mathcal{S}_{k-1}$. Choose $\mathcal{S}_k$ such that
the trajectory of $y_k$ under $\mathcal{S}_k$
visits each of the balls $B_1,\dots, B_K$.
Then the trajectory of
$c_k$ under $\mathcal{S}_1,\dots,\mathcal{S}_k$ 
visits each of the balls.

Let $u_1,\dots, u_N$ be the sequence of
directions given by $\mathcal{S}_1, \dots, \mathcal{S}_K$. 
For $x,y\in\SS^{d-1}$,
let $B_i$ and $B_j$ be the balls containing
$x$ and $y$, respectively.  By construction,
$$
F_{u_n}\dots F_{u_1}c_i \in B_j
$$
for some $n\le N$.
Since foldings are non-expansive, the triangle inequality implies
\begin{align*}
d(F_{u_n}\dots F_{u_1} x, y)
\le d(x,c_i) +
d(F_{u_n}\dots F_{u_1}c_i,c_j) + d(c_j,y) < \eps\,.
\end{align*}
This establishes Eq.~\eqref{eq:unif-dense}.
The desired infinite sequence $(u_n)$ is obtained
by concatenating the finite sequences constructed above for $\eps=2^{-j}$
with $j\ge 1$.
\end{proof}

\begin{proof}[Proof of Corollary~\ref{cor:radial}]
Let $\phi$ be a continuous function on $\RR^d$
such that $S_u\phi=\phi$
for all $u\in G$. 
Then $\phi\circ F_u\ge \phi$ for all $u\in G$,
that is, $\phi$ increases along trajectories.
We need to show that the restriction of $\phi$ to each
centered sphere $\{|x|=R\}$ is constant. By scaling, it suffices to
consider the case $R=1$.

Let $x,y\in\SS^{d-1}$ be given.
By Theorem~\ref{Thm:dense}
there exists an infinite sequence of directions $(u_n)$
such that the trajectory $(x_n)$
defined by Eq.~\eqref{eq:def-trajectory} is dense in $\SS^{d-1}$.
Choose a subsequence $(x_{n_k})$
that converges to~$y$. By monotonicity and continuity,
$$
\phi(x) \le \lim_{k\to\infty} \phi(x_{n_k}) = \phi(y)\,.
$$
Switching the  role of $x$ and $y$ yields the reverse inequality
$\phi(y)\le \phi(x)$.
We conclude that $\phi$ is constant on $\SS^{d-1}$.
\end{proof}

Theorem~\ref{Thm:dense} 
says that there exists a sequence of 
directions in a set $G$ (satisfying (C1) and (C2))
that generates dense trajectories for all starting points.
However, given $G$, it is not obvious how to
explicitly find a dense trajectory.
Different from reflections, periodic sequences of 
foldings never generate dense trajectories.

\begin{prop}
\label{prop:periodic}
Let $(u_n)_{n\ge 1}$
be a periodic sequence of directions, 
with $u_{n+p} = u_n$ for some integer $p$
and all $n\ge 1$.  Then Eq.~\eqref{eq:def-trajectory} 
has no dense trajectories, and at least
one trajectory is periodic. 
If, in addition, $G=\{u_1,\dots, u_p\}$
satisfies (C1) and (C2), then
the periodic trajectory is non-constant.
\end{prop}
\begin{proof}
Consider the composition
$F:= F_{u_p}\circ \dots \circ F_{u_1}$ as a
map from $H_{u_p}$ to itself.
Since $F$ is continuous and $H_{u_p}$ is homeomorphic to a ball
in $\RR^{d-1}$, by Brouwer's fixed point theorem
there exists a point $x\in H_{u_p}$ with $F(x)=x$. By construction,
the trajectory $(x_n)_{n\ge 0}$
of $x$ is periodic, and in particular not dense.
If $x_n'$ is another trajectory, then the
sequence $d(x_n,x_n')$ is non-increasing,
because folding maps are non-expansive.
Let $r=\lim d(x_n,x'_n)$. If $r=0$, then
the limit points of $(x'_n)$ are precisely
$x_1,\dots, x_p$. Otherwise,
let
$y$ be a limit point of $(x'_{kp})_{k\ge 1}$,
and let $(y_n)$ be the trajectory of $y$.
By construction, each point $y_n$
on the trajectory of $y$ is a limit
point of the subsequence $(x'_{n+kp})_{k\ge 1}$.
Therefore, $d(x_n,y_n)= r$
for all $n\ge 0$, that is, $(y_n)$ lies in the
union of the $d-2$-dimensional subspheres 
of radius $r$ centered at the points
$x_1,\dots, x_p$. Since this is a null set,
neither $(y_n)$ nor $(x_n')$ is dense in $\SS^{d-1}$.

For the last claim, suppose that
$(x_n)$ is constant. Then the singleton $\{x\}$
is positively invariant under~$G$.
By Lemma~\ref{lem:compact}, either (C1) or (C2)
fails.  \end{proof}

On the other hand, most trajectories are dense
under the assumptions
of Theorem~\ref{Thm:dense}.

\begin{prop} [Random walks are dense]\label{prop:RW-dense}
Let $(U_n)$ be an i.i.d.  sequence of random directions on
$\SS^{d-1}$ with distribution $\mu$.
If the support of $\mu$ satisfies {\rm (C1)}
and {\em (C2)} of Theorem~\ref{Thm:dense},
then almost surely the trajectory
in Eq.~\eqref{eq:def-RW}
is dense in $\SS^{d-1}$ for every starting point $x$.
\end{prop}

\begin{proof} Let $\eps>0$ be given.
For $n\ge 1$, let $A_{\eps,N}$ be the event that
for every pair of points $x,y\in\SS^{d-1}$, there
exists an integer $n$ such that
the random trajectory starting at $x$ intersects an
open $\eps$-neighborhood of $y$
within some segment $X_{n+1},\dots, X_{n+N}$.
By Theorem~\ref{Thm:dense}, there exists an $N<\infty$
and a sequence $(u_n)$ in the support of $\mu$
such that $A_{\eps,N}$ occurs on $X_1,\dots, X_N$.
Since finite segments of trajectories depend
continuously on the sequence of directions, the probability of
$A_{\eps, N}$ is strictly positive.  By the Borel-Cantelli lemma,
almost surely the event $A_{\eps, N}$ occurs infinitely often.
Since $\eps$ was arbitrary, the trajectory is almost surely dense.
\end{proof}

%%%%%%%%%%%%%%%%%%%%%%%%%%%%%%%%%%%%%%%%%%%
\section{Random walks and invariant measures}
\label{sec:random}

In this section, we prove the ergodicity results in 
Theorem~\ref{Thm:ergodic} and
Corollary~\ref{cor:Birkhoff}. For the construction
of the invariant measure, we need some more notation.

Consider for the moment a single
random direction $U$ in $\SS^{d-1}$
with probability distribution~$\mu$.
The random folding $F_U$
pulls a Borel function $\phi$ on $\SS^{d-1}$
back to
\begin{equation}
\label{eq:def-T}
T_\mu\phi(x) 
= E(\phi(F_U(x)) 
= \int_{\SS^{d-1}} \phi(F_u(x))\,d\mu(u)\,.
\end{equation}
The operator $T_\mu$ is linked to the action on measures defined in
Eq.~\eqref{eq:T-sharp} by the change-of-variables formula
$$
\int_{\SS^{d-1}} (T_\mu\phi)\, 
d\nu = \int_{\SS^{d-1}} \phi\, d(T_\mu\#\nu)\,.
$$
% which holds for every Borel measure $\nu$ and every
%Borel function $\phi$ on $\SS^{d-1}$ such that
%the integrals exist.
Clearly, $T_\mu$ is a positivity-preserving linear operator
that fixes constant functions.  Since folding maps are
non-expansive, $T_\mu$ preserves 
or improves the modulus of continuity of a continuous function.

The following lemma establishes a useful
monotonicity property for the extrema of~$\phi$ and 
$T_\mu\phi$.

\begin{lem} \label{lem:max}
Let $\mu$ be a probability measure
on $\SS^{d-1}$ whose support satisfies {\rm (C1)}.
Define $T_\mu$ by Eq.~\eqref{eq:def-T},
and let $\phi$ be a continuous function
on $\SS^{d-1}$ with $\max \phi=M$.
Then $\max T_\mu\phi\le M$, and for any $x\in\SS^{d-1}$
\[
T_\mu\phi(x)=M\quad \Longrightarrow\quad \phi(x)=M\,.
\]
The corresponding statements hold for the minimum of $\phi$.
\end{lem}

\begin{proof} Left $x\in\SS^{d-1}$.
Since $F_u(x)=x$ for all $u\in H_x$,
and $\phi(x)\le M$ for all $u$, we have
\begin{align*}
T_\mu\phi(x) &= \int _{H_x} \phi(x)\, d\mu(u) +
\int_{\SS^{d-1}\setminus H_x}
\phi(R_u x) \, d\mu(u)\\
&\le \phi(x)\mu(H_x) + M \mu(\SS^{d-1}\!\setminus\!H_x)\\
&\le M\,.
\end{align*}
Since $\mu(H_x)>0$ by Proposition~\ref{prop:cover},
equality implies $\phi(x)=M$.
\end{proof}

The random walk is associated with a canonical
Markov process on $\Omega=\prod_{n\ge 0} \SS^{d-1}$,
the space of sequences $(X_n)_{n\ge1}$
on $\SS^{d-1}$ endowed with the product
topology.
If $\Phi$ is a function on $\Omega$, we write
$$
E_x(\Phi((X_n)_{n\ge 0})) =
\int \Phi\left(
(F_{U_n}\dots F_{U_1}x)_{n\ge 0}\right) \,
d\mu^{\otimes}
$$
for the {\bf expected value} of $\Phi$ on the random walk $(X_n)$
starting at $X_0=x$. Here, $\mu^{\otimes}$
is the product measure that defines the distribution
of the sequence $(U_n)$.  Correspondingly, the probability of
an event $A\subset\Omega$ is denoted by
$P_x(A) = E_x(\mathbbm{1}_A)$.

The canonical Markov process is completely
determined by either of the operators~$T_\mu$ 
or $T_\mu\#$.
By the Markov property,
$$
E_x(\phi(X_n)) = (T_\mu)^n\phi(x)\,.
$$
The next result says that $(T_\mu)^n\phi$ converges uniformly
to a constant function.

\begin{prop} \label{prop:ergodic}
Let $\mu$ be a measure on $\SS^{d-1}$
whose support satisfies conditions~{\rm (C1)}
and {\rm (C2)} of
Theorem~\ref{Thm:dense}.
For every continuous function $\phi$ on $\SS^{d-1}$,
there exists a constant $\bar\phi$ such that
$$
\lim_{n\to\infty} E_x \left (\phi(X_n)\right) = \bar\phi
$$
uniformly in $x$. Moreover,
$$
\min \phi \le \bar\phi \le \max\phi\,.
$$
Both inequalities are strict unless
$\phi$ is constant.
\end{prop}

\begin{proof} Let $||\cdot||$ denote the sup-norm on the
space of continuous functions.
Set $\phi_n(x)= E_x(\phi(X_n))$.
Since $\phi_n= (T_\mu)^n\phi$, its
norm $||\phi_n||$ decreases monotonically and
the modulus of continuity of $\phi_n$ improves with~$n$.
By the Arzel\`a-Ascoli theorem,
there exists a subsequence $(\phi_{n_k})_{k\ge 1}$
that converges uniformly to some limiting function,
$\bar \phi$.  After passing to a further subsequence, we may
assume that the sequence of gaps $n_k-n_{k-1}$ increases strictly
with $k$.

We want to show that $\bar\phi$ is constant.
Consider the sequence
$(\psi_n)$ defined by $\psi_n=(T_\mu)^n\bar \phi$.
Clearly,
$$
||\psi_m-\phi_{m+n}|| = ||(T_\mu)^m(\bar\phi-\phi_n)||
\le ||\bar\phi- \phi_n||
$$
for all $m,n\ge 0$. We use the
triangle inequality and set
$m=n_k-n_{k-1}$, $n=n_k$ to obtain the bound
\begin{align*}
||\psi_{n_k-n_{k-1}} - \bar\phi||
&\le
|| \psi_{n_k-n_{k-1}} - \phi_{n_k}||+ || \phi_{n_k}-\bar\phi||\\
&\le
|| \bar\phi - \phi_{n_{k-1}}||+ || \phi_{n_k}-\bar\phi||\,,
\end{align*}
which converges to zero by the choice of the subsequence $(n_k)$.
It follows that the subsequence $\psi_{n_k-n_{k-1}}$
converges uniformly to $\bar\phi$.

Let $M=\max\bar\phi$.
Since $\max\psi_n$ is non-increasing in $n$, and
a subsequence converges to $\psi_0=\bar\phi$,
we must have $\max\psi_n=M$ for all $n\ge 0$.
By Lemma~\ref{lem:max}, the
sets $A_n=\{x: \psi_n(x)=M\}$ form a decreasing
chain. Their
intersection $A=\bigcap A_n$ is a non-empty compact set
that is positively
invariant under $F_u$ for $\mu$-a.e. $u\in\SS^{d-1}$.
By Lemma~\ref{lem:01}, $A=\SS^{d-1}$. But this says that
$\psi_n\equiv M$ for all $n$, and the same holds for their limit $\bar\phi$.
Since the sequence
$||\phi_n-\bar\phi||$ is non-increasing, the convergence of the
full sequence follows from the convergence of the subsequence.

Clearly, $\min \phi\le \bar\phi\le \max\phi$,
since $\min \phi_n$ is non-decreasing and
$\max\phi_n$ is non-increasing.
Let $M=\max \phi$, and consider the decreasing
chain of compact subsets $A_n=\{x:\phi_n(x)=M\}$.
Since the intersection $A=\bigcap A_n$ is compact and
positively invariant, by Lemma~\ref{lem:compact}, it is
either equal to $\SS^{d-1}$ or empty. In the first case,
$\phi\equiv M$ is constant. In the second case, by compactness,
$A_n$ is empty and thus $\max\phi_n<M$ for some sufficiently
large $n$. By monotonicity, $\bar\phi\le \max \phi_n <M$.
The same argument shows that $\min \phi<\bar\phi $ unless $\phi$ is
constant.
\end{proof}

\begin{proof}[Proof of Theorem~\ref{Thm:ergodic}] 
Let $\phi$ be a continuous function on $\SS^{d-1}$,
and let $(X_n)$ be the random walk defined
by Eq.~\eqref{eq:def-RW}.
By Proposition~\ref{prop:ergodic}, 
$ E_x \left (\phi(X_n)\right)$ converge 
to a constant function, $\bar \phi$, uniformly in $x$.
It is a fact of Ergodic Theory
that this uniform convergence is equivalent to the unique 
ergodicity of the Markov chain~\cite[Theorem 4.10]{EW}.  
We show the part of the proof that we need here.

The map $\phi\mapsto \bar\phi$ is linear and continuous
with respect to the topology of uniform convergence,
and its value on nonnegative functions
is nonnegative. By the Riesz-Markov theorem,
there is a unique regular Borel measure
$\rho$ such that
$$\bar\phi = \int_{\SS^{d-1}} \phi \, d\rho\,.
$$
Since the constant function $\phi\equiv 1$ is mapped to
itself, $\rho$ is a probability
measure.

We next verify that $\rho$ is invariant.
For every continuous function $\phi$ on $\SS^{d-1}$,
$$
\int_{\SS^{d-1}} \phi \, d(T_\mu\#\rho) 
= \int_{\SS^{d-1}} (T_\mu\phi) \, d\rho
= \overline{T_\mu\phi}\,.
$$
Since 
$$
\overline{T_\mu\phi} = \lim_{n\to\infty} E_x(T_\mu\phi(X_n))
= \lim_{n\to\infty} E_x(\phi(X_{n+1}) = \bar\phi
= \int_{\SS^{d-1}} \phi\, d\rho\,,
$$
the measure $T_\mu\#\rho$ 
represents the same distribution as
$\rho$. By uniqueness,  $T_\mu\#\rho=\rho$.

It remains to show that
the support of $\rho$ is $\SS^{d-1}$.
Given an arbitrary non-empty open
set $A\subset\SS^{d-1}$,
let $\phi$ be a nonnegative
continuous function supported on $A$ that takes values in $[0,1]$
and does not vanish identically.
Then $\rho(A)\ge \bar\phi>0$ by the last
part of Proposition~\ref{prop:ergodic}.
Since $A$ was arbitrary, the proof is complete.
\end{proof}

As an immediate consequence of the uniform convergence proved
in Proposition~\ref{prop:ergodic}, we obtain
the {\bf mixing} property
$$
\lim_{n\to\infty} \int_{\SS^{d-1}}
E_x(\phi(X_n))\,\psi(x)\,d\rho(x) = \bar\phi \bar\psi
$$
for every pair of continuous functions $\phi, \psi$ on
$\SS^{d-1}$.  

\begin{proof}[Proof of Corollary~\ref{cor:Birkhoff}]
Let $\mu$ and $\rho$ be as in Theorem~\ref{Thm:ergodic}.
We will show that
for every $\rho$-integrable function on $\SS^{d-1}$.
\begin{equation}
\label{eq:Birkhoff-phi}
\lim_{N\to\infty}
\frac1N \sum_{k=1}^N \phi(X_k) = \int_{\SS^{d-1}}
\phi(y)\, d\rho(y)
\end{equation}
almost surely for $\rho$-almost every $x\in\SS^{d-1}$.
The claim then follows by setting
$\phi=\mathbbm{1}_A$.

The proof of Eq.~\eqref{eq:Birkhoff-phi}
calls for a standard application of
Birkhoff's ergodic theorem to the canonical
Markov chain associated with the random walk.
See for example~\cite[Chapter 6]{Varadhan}.
The invariant measure induces
a probability measure $\rho^*$ on $\Omega$ by
$$
\rho^*(A) = \int_{\SS^{d-1}} P_x(A)\, d\rho(x)\,.
$$
This measure is invariant under the left shift
$L( (X_n)_{n\ge 0} ) = (X_{n+1})_{n\ge 0}$,
and every shift-invariant subset $A\subset\Omega$
has $\rho^*(A)=0$ or $\rho^*(A)=1$.
By Birkhoff's theorem,
$$
\lim_{N\to\infty} \frac{1}{N}
\sum_{k=1}^N \Phi(L^k(X_n)_{n\ge 0}) = \int_\Omega \,
\Phi \,d\rho^*
$$
for every $\rho^*$-integrable function $\Phi$ on $\Omega$
and for $\rho^*$-almost every sequence
$(X_n)_{n\ge 0}$. In the case
where $\Phi((X_n)_{n\ge 0}) = \phi(X_0)$ depends only on 
the initial point, we have
$$
\Phi(L^k(X_n)_{n\ge 0}) = \phi(X_k)\,,\quad
\int_\Omega \Phi \, d\rho^*=\int_{\SS^{d-1}}\phi\, d\rho\,,
$$
which yields Eq.~\eqref{eq:Birkhoff-phi}
except for sequences $(X_n)$ in a set
$B\subset\Omega$ of $\rho^*$-measure zero.
By definition of $\rho^*$ and Fubini's theorem,
$P_x(B)=0$ for $\rho$-almost every $x\in\SS^{d-1}$.
\end{proof}

If $\phi$ is a continuous function on $\SS^{d-1}$,
then the functions
$\frac{1}{N} \sum_{n=1}^N\phi(X_n)$
are uniformly equicontinuous in $x$ for all
$N\ge 1$ and every sequence of directions
$(U_n)$, see Eq.~\eqref{eq:def-RW}.
Since $\SS^{d-1}$ is separable, it follows that
Eq.~\eqref{eq:Birkhoff-phi}
almost surely holds for every $x\in\SS^{d-1}$.

%%%%%%%%%%%%%%%%%%%%%%%%%%%%%%%%%%%%%%%%%%%
\section{Properties of the invariant measure}
\label{sec:inv-meas}

Finally, we study the properties of $\rho$. We find that
$\rho$ is generally not the uniform measure on the sphere.

\begin{prop}
\label{prop:even}
Under the assumptions of Theorem~\ref{Thm:ergodic},
the invariant measure is uniform 
on $\SS^{d-1}$, if and only if $\mu$ is even under 
$u\mapsto -u$.
\end{prop}

\begin{proof} 
Recall that $\sigma$ denotes the uniform probability
measure on $\SS^{d-1}$.
We want to show that
$$
T_\mu \#\sigma = \sigma
\quad \Longleftrightarrow \quad
\mu(A)=\mu(-A)\ \text{for all Borel sets}\ A\subset\SS^{d-1}\,.
$$

$\Rightarrow$: Suppose 
$T_\mu\#\sigma = \sigma$.
Then, for every $v\in \SS^{d-1}$ and $\eps>0$,
$$
\int_{\SS^{d-1}} \sigma(F_u^{-1}(B_\eps(v)))\, d\mu(u) = \sigma(B_\eps)\,.
$$
By definition of the folding map,
the integrand is given by
$$
\sigma(F_u^{-1}(B_\eps(v))) = 2\sigma\left(B_\eps(v)\cap H_u\right)\,.
$$
We divide by $\sigma(B_\eps)$ and take $\eps\to 0$,
$$
\lim_{\eps\to 0^+}
\frac{\sigma(F_u^{-1}(B_\eps(v))} {\sigma(B_\eps)}
= \lim_{\eps\to 0^+}
\frac{2\sigma(B_\eps(v)\cap H_u)} {\sigma(B_\eps)}
= \left \{\begin{array}{ll}
2 \quad &\text{if}\ u\cdot v >0 \,, \\
1 &\text{if} \ u\cdot v = 0\,,\\
0 &\text{otherwise}\,.
\end{array}\right.
$$
Integrating both sides over $H_v$ and using that
$T_\mu \#\sigma=\sigma$, we obtain
by dominated convergence
$$
1
= \lim_{\eps\to 0^+} \int_{\SS^{d-1}}
\frac{\sigma(F_u^{-1}(B_\eps(v)))} {\sigma(B_\eps)}\, d\mu(u)
= 2\mu(H_v) + \mu(\partial H_v)\,.
$$
It follows that $\mu(H_v)=\mu(H_{-v})$ for all $v\in\SS^{d-1}$.
By Lemma~\ref{lem:even}, which is proved below, $\mu$ is even.

$\Leftarrow$: Conversely, if $\mu(A)=\mu(-A)$
for all $A\subset\SS^{d-1}$, then
\begin{align*}
T_\mu\#\sigma(A) 
&= \int \sigma(F_u^{-1}(A))\, d\mu(u)\\
&=\frac12 \int \sigma(F_u^{-1}(A)) + \sigma(F_{-u}^{-1}(A))\, d\mu(u)\\
&=\frac12 \int \sigma(A)+ \sigma(R_uA)\, d\mu(u)\\
&=\sigma(A)\,.
\end{align*}
In the second line, we have used that $\mu$ is
even to change the variable $u$ to $-u$
in half of the integral. The third line follows, since
$F_{\pm u}^{-1}(A)$  contains a copy
of $A\cap H_{\pm u}$ together with its mirror image.
In the last step we have exploited the
reflection invariance of $\sigma$.
\end{proof}

\begin{lem} \label{lem:even} 
Let $\mu$ be a regular Borel measure on
$\SS^{d-1}$. If
$$
\mu(H_u)=\mu(H_{-u})
$$
for all hemispheres $H_u$ with $u\in\SS^{d-1}$, then $\mu$ is even.
\end{lem}

\begin{proof}
Assume for the moment that $\mu$ is absolutely continuous
with respect to 
the uniform measure $\sigma$ on $\SS^{d-1}$, 
with a smooth density $m(x)$, and
consider the expansion in spherical harmonics
$$
m(x)=\sum_{k\ge o} Y_k(x)\,.
$$
By the Funk-Hecke formula~\cite[Theorem 9.7.1]{AAR}, 
there exist constants $\gamma_k$ such that
\begin{equation}
\label{eq:FH}
\int_{H_u} Y_k(x) \, d\sigma(x) 
= \gamma_k Y_k(u)
\end{equation}
for all $u\in\SS^{d-1}$.
In particular,
$$
\mu(H_u) = \int_{H_u} m(x)\, d\sigma(x) = \sum_{k\ge 0} \gamma_k Y_k(u)\,.
$$
Since $\mu(H_u)=\mu(H_{-u})$
for all $u\in\SS^{d-1}$ and $Y_k(u)=(-1)^kY_k(-u)$,
the component $\gamma_kY_k$ must vanish for each odd $k$ 
by the uniqueness of the expansion.
We claim that $\gamma_k\ne 0$ for odd $k$, forcing $Y_k=0$.

A key point of the Funk-Hecke formula is
that the constant $\gamma_k$ depends only on the order ($k$)
of the harmonic polynomial.  
It can be computed from Eq.~\eqref{eq:FH}
by setting $u=e_d$, and choosing
$Y_k$ to be the normalized
zonal harmonic $Z_k$, whose restriction to
$\SS^{d-1}$ depends only on the last variable $x_d$.
For $k>0$, we use that $Z_k$ is an eigenfunction of 
the spherical Laplacian with an eigenvalue $-\lambda_k<0$
to obtain
$$
\int_{\{x_d>0\}} Z_k \, d\sigma
= -\lambda_k^{-1}
\int_{\{x_d>0\}} \Delta Z_k\, d\sigma\\
= \lambda_k^{-1}
\int_{\{x_d=0\}} -\partial_{x_d} Z_k \, d\sigma \,.
$$
In the second step, we have applied the divergence theorem
on the sphere. Since $Z_k$ depends
only on the variable $x_d$, 
the normal derivative that appears in the
last integral is constant on the equatorial
sphere $\{x_d=0\}$.

For $k$ odd, $Z_k$ vanishes on $\{x_d=0\}$.
Since the eigenvalue-eigenvector
equation for $Z_k$ is a homogeneous linear
second-order ordinary differential equation
in $x_d$, the normal derivative cannot vanish
simultaneously, and so the integral is non-zero.
It follows that
$$
\gamma_k= \frac{1}{Z_k(e_d)} \int_{\{x_d>0\}} Z_k\, d\sigma
\ne 0\qquad (k\ \text{odd})\,.
$$
We conclude that
$$
\mu(A) = \sum_{k\ge 0\ \text{even}} \int_A Y_k\, d\sigma = \mu(-A)
$$
for all Borel sets $A\subset\SS^{d-1}$,
proving the claim when $\mu$ has a smooth density.

Otherwise, we approximate it with smooth measures
$\mu_\eps$, defined by
$$
\mu_\eps(A) = \int_{SO(d)} \mu(QA)\psi_\eps(Q)\, d\sigma(Q)\,,
$$
where $\psi_\eps$ is a smooth probability density
supported on an $\eps$-neighborhood of the identity in $SO(d)$,
and $\sigma$ is the uniform measure. 
Let $u\in\SS^{d-1}$. Since $Q(H_u)=H_{Qu} = -QH_{-u}$
by linearity, the assumption on $\mu$ implies that
$$
\mu(QH_u) = \mu(QH_{-u}) \qquad (Q\in SO(d))\,.
$$
Therefore, $\mu_\eps(H_u)=\mu_\eps(H_{-u})$ for all $u\in\SS^{d-1}$.
By the first part of the proof,
$\mu_\eps$ is even.  Taking $\eps\to 0$ we see that $\mu$ is even as well.
\end{proof}

We conclude with some open problems.
\begin{qu} Given a measure $\mu$ on $\SS^{d-1}$
whose support satisfies {\rm (C1)} and {\rm (C2)}, let 
$\rho$ be the invariant measure  from Theorem~\ref{Thm:ergodic}. 
Is it true that
$$
\rho(A)=0 \quad \Longleftrightarrow \quad \sigma(A)=0\,?
$$
\end{qu}

We suspect that the answer is affirmative.
To motivate this, we write Eq.~\eqref{eq:T-sharp} as
$$
T_\mu\#\nu(A) 
= \int_{\SS^{d-1}}\int_{\SS^{d-1}} \mathbbm{1}_A(F_u(x))\,
d\mu(u)d\nu(x)\,,
$$
and observe that $T_\mu\#\nu$ is absolutely continuous with
respect to $\sigma$,
whenever either $\mu$ or $\nu$ is absolutely continuous.
In particular, the absolutely
continuous component of the invariant measure
$\rho$ is itself invariant. By the uniqueness
part of Theorem~\ref{Thm:ergodic},
$\rho$ is either absolutely continuous or purely singular.
If $\mu$ has an absolutely continuous component
in its Lebesgue decomposition, then $\rho$ is
absolutely continuous. If
$\mu$ is singular, nothing is known.
If $\rho$ were singular as well, then the
trajectories of the random walk would accumulate
on a set of measure zero.

Another problem is to characterize how quickly the
state of the Markov chain associated
with the random walk in Eq.~\eqref{eq:def-RW} converges to the
steady-state.

\begin{qu} Given a measure $\mu$ on $\SS^{d-1}$
whose support satisfies {\rm (C1)} and {\rm (C2)}, 
fix $x\in\SS^{d-1}$, and
define the random walk $X_n$ by Eq.~\eqref{eq:def-RW}.
At what asymptotic rate does the distribution of
$X_n$
converge to $\rho$ as $n\to\infty$? Is there a cut-off phenomenon?
\end{qu}

One approach is to analyze the convergence
of $\phi_n= (T_\mu)^n\phi$ to its limit 
$\bar\phi = \int\phi\, d\rho$ for continuous functions $\phi$
on $\SS^{d-1}$.  In the case where $\mu=\sigma$ is the uniform measure
on the sphere, computations similar to
those in~\cite[Proposition 5.2]{BF}
and~\cite{Bur} suggest that the difference
$||\phi_n-\bar\phi||$
decreases with $n^{-1}$ in the number of steps.
Can this rate of convergence
be improved by a judicious choice
of $\mu$, perhaps supported on a finite set?

\section*{Acknowledgments} We thank
Dima Jakobson, Neal Sloane, and Qin Deng for
helpful discussions.
This work was supported in part by
an NSERC Discovery Grant (A.B.)
and an Ontario Graduate Fellowship (G.R.C.).

% You may incorporate your references as follows in your main tex file.
% Using BibTex is not recommended but can be handled.

\medskip
% The data information below will be filled by AIMS editorial staff
%Received xxxx 20xx; revised xxxx 20xx.
\medskip

\end{document}